\documentclass[11pt]{article}
\usepackage{amsmath,amsfonts,amsopn,amsthm,amssymb}
\usepackage{hyperref}
\usepackage[T1]{fontenc}
\usepackage{mathptmx}

\theoremstyle{remark}
\newtheorem{remark}{Remark}[section]
\theoremstyle{plain}
\newtheorem{theorem}[remark]{Theorem}
\newtheorem{proposition}[remark]{Proposition}
\newtheorem{lemma}[remark]{Lemma}

\theoremstyle{definition}
\newtheorem{definition}[remark]{Definition}

\let\ge=\varepsilon

\newcommand{\slap}{\left(-\Delta\right)^s}

\begin{document}
 
 \title{On fractional Schr\"{o}dinger equations in \(\mathbb{R}^N\) without the
Ambrosetti-Rabinowitz condition}
\author{Simone Secchi\thanks{Partially supported by PRIN 2009 ``Teoria dei punti critici
e metodi perturbativi per equazioni differenziali nonlineari''.} \\ Dipartimento di matematica ed applicazioni \\ Universit\`a di
Milano-Bicocca}
\date{\today}

\maketitle

\begin{abstract}
 In this note we prove the existence of radially symmetric solutions for a
class of fractional Schr\"{o}dinger equation in \(\mathbb{R}^N\) of the form
\begin{equation*}
 \slap u + V(x) u = g(u),
\end{equation*}
where the nonlinearity $g$ does not satisfy the usual Ambrosetti-Rabinowitz
condition. Our approach is variational in nature, and leans on a Pohozaev
identity for the fractional laplacian.
\end{abstract}

\vspace{5mm}



\section{Introduction}

Fractional scalar field equations have attracted much attention in recent
years, because of their relevance in obstacle problems, phase transition,
conservation laws, financial market. Strictly speaking, these
equations are not partial differential equations, but rather integral
equations. Their main feature, and also their main difficulty, is that they are
strongly \emph{non-local}, in the sense that the leading operator takes care of
the behavior of the solution in the whole space. This is in striking contrast
with the usual elliptic partial differential equations, which are governed by
\emph{local} differential operators like the laplacian.

In the present paper we deal with a class of fractional scalar field equations
with an external potential,
\begin{equation}\label{eq:1}
\slap u + V(x) u = g(u),  \qquad x \in \mathbb{R}^N, 
\end{equation}
which we will briefly call \emph{fractional Schr\"{o}dinger equation}. The
operator $\slap$ is a non-local operator that we may describe in several ways.
Postponing a short discussion about this operator to the next section, we can
think that the fractional laplacian $\slap$ of order $s \in (0,1)$  is the
pseudodifferential operator with symbol $|\xi|^s$, i.e.
\begin{equation*}
 \slap u = \mathcal{F}^{-1} \left( |\xi|^{2s} \mathcal{F}u \right),
\end{equation*}
$\mathcal{F}$ being the usual Fourier transform in $\mathbb{R}^N$. The
non-local property of the fractional laplacian is therefore clear: $\slap u$
need not have compact support, even if $u$ is compactly supported.

It is known, but not completely trivial, that $\slap$ reduces to the standard
laplacian $-\Delta$ as $s \to 1$ (see \cite{DiNezza2012521}). In the sequel we will identify $\slap$ with
$-\Delta$ when $s=1$.

When $s=1$, equations like (\ref{eq:1}) are called Nonlinear
Schr\"{o}dinger Equations (NLS for short), and we do not even try to
review the huge bibliography. On the contrary, the situation seems to
be in a developing state when $s < 1$. A few results have recently
appeared in the literature. In \cite{Dipierro2012} the authors prove the
existence of a nontrivial, radially symmetric, solution to the equation
\[
\slap u + u = |u|^{p-1}u \qquad \text{in $\mathbb{R}^N$}
\]
for subcritical exponents $1<p<(N+2s)/(N-2s)$.

In \cite{Secchi2012a, Secchi2012} the author proves some existence results for
fractional Schr\"{o}dinger equations, under the assumption that the
nonlinearity is either of perturbative type or satisfies the
Ambrosetti-Rabinowitz condition (see below).

In the present paper, we will solve (\ref{eq:1}) under rather weak assumptions on $g$, which are comparable to those in \cite{MR695535}. The presence of the fractional operator $\slap$ requires some technicalities about the regularity of weak solutions and the compactness of the embedding of radially symmetric Sobolev functions. Since the statement of our results needs some preliminaries on fractional Sobolev spaces, we present a very quick survey of their main definitions and properties.

We will follow closely the ideas developed by Azzollini \emph{et al.} in \cite{MR2542090} for the Schr\"{o}dinger equation and then extended to other situations like the Schr\"{o}dinger-Maxwell equations (\cite{MR2595202}) and Schr\"{o}dinger systems (\cite{MR2600460}). Several modifications will be necessary to deal with the non-local features of our problem.

\section{A quick review of the fractional laplacian}

As we said in the introduction, different definitions can be given of
the fractional Schr\"{o}dinger operator $\slap$, but in the end they
all differ by a multiplicative constant. In this section we offer a
rather sketchy review of this operator, and we refer for example to
\cite{DiNezza2012521} for a more exhaustive discussion.

In the rest of this section, $s$ will denote a fixed number, $0<s<1$.

\begin{definition}
Given $p \in [1,+\infty)$, the Sobolev space $W^{s,p}(\mathbb{R}^N)$ is the space defined by
\begin{equation*}
  W^{s,p}(\mathbb{R}^N) = \left\{ u \in L^p(\mathbb{R}^N) \mid \frac{|u(x)-u(y)|}{|x-y|^{\frac{n}{p}+s}} \in L^p (\mathbb{R}^N \times \mathbb{R}^N) \right\}.
\end{equation*}
This space is endowed with the natural norm
\begin{equation*}
\|u\|_{W^{s,p}} = \left( \int_{\mathbb{R}^N} |u(x)|^p\, dx + \int_{\mathbb{R}^N} \int_{\mathbb{R}^N} \frac{|u(x)-u(y)|^p}{|x-y|^{n+sp}}\, dx \, dy \right)^{\frac{1}{p}},
\end{equation*}
while
\begin{equation*}
[u]_{W^{s,p}} =  \left( \int_{\mathbb{R}^N} \int_{\mathbb{R}^N} \frac{|u(x)-u(y)|^p}{|x-y|^{n+sp}}\, dx \, dy \right)^{\frac{1}{p}}
\end{equation*}
is the \emph{Gagliardo (semi)norm} of $u$.
\end{definition}

For the reader's convenience, we recall the main embedding results for
this class of fractional Sobolev spaces.

\begin{theorem}
\begin{itemize}
\item[(a)] Let $0<s<1$ and $1 \leq p < +\infty$ be such that $sp<N$. Then there exists a constant $C=C(N,p,s)>0$ such that 
\[
\|u\|_{L^{p^\star}} \leq C \|u\|_{W^{s,p}}
\]
for every $u \in W^{s,p}(\mathbb{R}^N)$. Here
\[
p^\star = \frac{Np}{N-sp}
\]
is the ``fractional critical exponent''. Moreover, the embedding $W^{s,p}(\mathbb{R}^N) \subset L^q(\mathbb{R}^N)$ is locally compact whenever $q<p^\star$.
\item[(b)] Let $0<s<1$ and $1 \leq p < +\infty$ be such that $sp=N$. Then there exists a constant $C=C(N,p,q,s)>0$ such that 
\[
\|u\|_{L^q} \leq C \|u\|_{W^{s,p}}
\]
for every $u \in W^{s,p}(\mathbb{R}^N)$ and every $q \in [p,+\infty)$.
\item[(c)]  Let $0<s<1$ and $1 \leq p < +\infty$ be such that $sp>N$. Then there exists a constant $C=C(N,p,s)>0$ such that 
\[
\|u\|_{C_{\mathrm{loc}}^{0,\alpha}} \leq C \|u\|_{W^{s,p}}
\]
for every $u \in W^{s,p}(\mathbb{R}^N)$ and $\alpha = (sp-N)/p$.
\end{itemize}
\end{theorem}

When $p=2$, the Sobolev space $W^{s,2}(\mathbb{R}^N)$ turns out to be a Hilbert space that we can equivalent describe by means of the Fourier transform. Indeed, it is well-known that
\[
W^{s,2}(\mathbb{R}^N) = \left\{ u \in L^2(\mathbb{R}^N) \mid \int_{\mathbb{R}^N} \left( 1+|\xi|^{2s} \right) |\mathcal{F}u(\xi)|^2 \, d\xi < +\infty \right\} .
\]
It will be convenient to denote $W^{s,2}(\mathbb{R}^N)$ by $H^s(\mathbb{R}^N)$. 
\begin{definition}
  If $u$ is a rapidly decreasing $C^\infty$ function on
  $\mathbb{R}^N$, usually denoted by $u \in \mathcal{S}$, the
  fractional laplacian $\slap$ acts on $u$ as
\begin{align}
\slap u(x) &= C(N,s) P.V. \int_{\mathbb{R}^N} \frac{u(x)-u(y)}{|x-y|^{n+2s}}\, dy \\
&= C(N,s) \lim_{\ge \to 0+} \int_{\mathbb{R}^N \setminus B(0,\ge)} \frac{u(x)-u(y)}{|x-y|^{n+2s}}\, dy
\end{align}
The costant $C(N,s)$ depends only on the space dimension $N$ and on the order $s$, and is explicitly given by the formula
\[
\frac{1}{C(N,s)} = \int_{\mathbb{R}^N} \frac{1-\cos \zeta_1}{|\zeta|^{n+2s}}\, d\zeta.
\]
\end{definition}
It can be proved (see \cite[Proposition 3.3 and Proposition 3.4]{DiNezza2012521}) that 
\[
\slap u = \mathcal{F}^{-1} \left( |\xi|^{2s} \mathcal{F}u \right)
\]
and that
\[
[u]_{H^s}^2 = \frac{2}{C(N,s)} \int_{\mathbb{R}^N} |\xi|^{2s} \left| \mathcal{F}u(\xi)\right|^2 \, d\xi.
\]
Moreover, 
\[
[u]_{H^s}^2 = \frac{2}{C(N,s)} \left\| (-\Delta)^{\frac{s}{2}} u \right\|_{L^2}^2.
\]
As a consequence, the norms on $H^s(\mathbb{R}^N)$
\begin{align*}
u &\mapsto \|u\|_{W^{s,2}} \\
u &\mapsto \left( \|u\|_{L^2}^2 + \int_{\mathbb{R}^N} |\xi|^{2s} |\mathcal{F}u(\xi)|^2\, d\xi\right)^{\frac{1}{2}} \\
u &\mapsto \left( \|u\|_{L^2}^2 + \|(-\Delta)^{\frac{s}{2}}u\|_{L^2}^2 \right)^{\frac{1}{2}}
\end{align*}
are all equivalent.
%

\bigskip

A different characterization of the fractional laplacian was given by
Caffarelli and Silvestre in \cite{MR2354493} and runs as follows. Given a function $u$, consider its \emph{extension} $U \colon \mathbb{R}^N \times (0,+\infty) \to \mathbb{R}$ such that
\begin{equation*}
\left\{
\begin{array}{ll}
\operatorname{div} \left( t^{1-2s}\nabla U \right) =0 &\text{in $\mathbb{R}^N \times (0,+\infty)$} \\
U(x,0)=u(x) &\text{in $\mathbb{R}^N$}.
\end{array}
\right.
\end{equation*}
Then there exists a positive constant $C$ such that
\[
\slap u(x) = -C \lim_{t \to 0+} \left( t^{1-2s} \frac{\partial U}{\partial t}(x,t) \right).
\]
Moreover
\[
\int_{\mathbb{R}^N} |\xi|^{2s} |\hat{u}(\xi)|^2 \, d\xi = C \int_{\mathbb{R}^N \times (0,+\infty)} |\nabla U|^2 t^{1-2s}\, dx \, dt.
\]
Hence the fractional laplacian can also be considered as a ``local''
operator in an ``augmented space''. We will not directly use this
characterization, in our paper. However, regularity theorems for the fractional laplacian are often easier to prove with this characterization.

\section{Main results}

Let us get back to our equation (\ref{eq:1}). We will try to solve it
in the natural Hilbert space $H^s(\mathbb{R}^N)$, where (weak)
solutions correspond to critical points of the Euler functional $I
\colon H^s(\mathbb{R}^N) \to \mathbb{R}$ defined by
\begin{equation}
I(u) = \frac{1}{2} \int_{\mathbb{R}^N} |\xi|^{2s} |\hat{u}(\xi)|^2 \, d\xi + \frac{1}{2}\int_{\mathbb{R}^N} V(x) |u(x)|^2 \, dx - \int_{\mathbb{R}^N} G(u(x))\, dx.
\end{equation}
Here we have denoted $\hat{u}=\mathcal{F}u$ and $G(s) = \int_0^s g(t)\, dt$.

The loss of compactness associated to (\ref{eq:1}) is not trivial, in
the sense that Palais-Smale sequences for the functional~$I$ need not
converge (up to subsequences). In particular so-called Ambrosetti-Rabinowitz condition
\begin{equation} \label{eq:23}
\mu \int_{\mathbb{R}^N} G(u(x))\, dx \leq \int_{\mathbb{R}^N} g(u(x))u(x)\, dx
\end{equation}
for some $\mu>2$ is often assumed to deduce the boundedness of Palais-Smale sequences.

When $V \colon \mathbb{R}^N \to \mathbb{R}$ is constant (say $V=1$)
and $s=1$, Berestycki and Lions proved in \cite{MR695535} that
non-trivial, radially symemtric solutions to (\ref{eq:1}) exist under
mild assumptions on $g$, and the Ambrosetti-Rabinowitz condition is
not necessary. Their approach is based on a constrained minimization
that we cannot expect to work when $V$ is non-constant. 

To deal with this more general case for the fractional Schr\"{o}dinger
operator we will follow the ideas of Azzollini \emph{et al.}
\cite{MR2542090} to get both existence and non-existence results for (\ref{eq:1}).

\bigskip

Let us fix the standing assumptions of our paper. The nonlinearity $g$ will satisfy
\begin{itemize}
\item[(g1)] $g \colon \mathbb{R} \to \mathbb{R}$ is of class $C^{1,\gamma}$ for some $\gamma > \max\{0,1-2s\}$, and odd;
\item[(g2)] $-\infty < \liminf_{t \to 0+} \frac{g(t)}{t} \leq
  \limsup_{t \to 0+} \frac{g(t)}{t} = -m <0$;
\item[(g3)] $-\infty< \limsup_{t \to +\infty} \frac{g(t)}{t^{2^\star-1}} \leq 0$;
\item[(g4)] for some $\zeta>0$ there results $G(\zeta)=\int_0^\zeta g(t)\, dt >0$.
\end{itemize}
\begin{remark}
  Replacing $2^\star$ with $2^*=2N/(N-2)$, these are the same assumptions
  of \cite{MR695535}. In particular there is no superlinearity
  requirement at infinity and no Ambrosetti-Rabinowitz condition.
  
  The regularity of $g$ is higher than in \cite{MR2542090} or \cite{MR695535}, and this seems to be due to the more demanding assumptions for ``elliptic'' regularity in the framework of fractional operators, see \cite{Cabre2010}.
\end{remark}
On the other hand, the potential $V$ will satisfy
\begin{itemize}
\item[(V1)] $V \in C^1 (\mathbb{R}^N,\mathbb{R})$, $V(x)\geq 0$ for
  every $x \in \mathbb{R}^N$ and this inequality is strict at some
  point;
\item[(V2)] $\|\max\{\langle \nabla V(\cdot),\cdot \rangle,0\}\|_{L^{N/2s}} < 2S$;
\item[(V3)] $\lim_{|x| \to+\infty} V(x)=0$;
\item[(V4)] $V$ is radially symmetric, i.e. $V(x)=V(|x|)$.
\end{itemize}
Here $S$ is the best Sobolev constant for the critical embedding, viz.
\[
S = \inf_{\substack{u \in \dot{H}^s(\mathbb{R}^N) \\ u \neq 0}}
\frac{\|(-\Delta)^{\frac{s}{2}}u\|_{L^2}^2}{\|u\|_{L^{2^\star}}}
\]
and $\dot{H}^s(\mathbb{R}^N)$ is the \emph{homogeneous Sobolev space}
consisting of the measurable functions $u$ such that
$\int_{\mathbb{R}^N} |(-\Delta)^{\frac{s}{2}}u|^2 < +\infty$. See
\cite{MR2064421} for a discussion about $S$ and its minimizers.
We can formulate our main result about existence of solutions of
equation (\ref{eq:1}).
\begin{theorem} \label{thm:existence}
  Assume that $1/2<s<1$, $g$ satisfies (g1--4) and $V$ satisfies (V1--4). Then
  there exists a nontrivial solution $u \in H^s(\mathbb{R}^N)$ of
  equation (\ref{eq:1}), and this solution is radially symmetric.
\end{theorem}
\begin{remark}
  As we shall see in the next section, weak solutions of (\ref{eq:1})
  have additional regularity. We will need this fact to prove a
  Pohozaev identity for our equation.
\end{remark}
We will comment later on the restriction $1/2<s<1$. If we want to
remove this condition, we need to be more precise about the behavior
of the nonlinearity $g$.
\begin{theorem} \label{thm:existence2}
  Assume that $0<s<1$, that $V$ satisfies (V1--4) and that $g$ satisfies (g1), (g2), (g4) and
\begin{itemize}
\item[(g3)'] for some $q<2^*$, $|g(t)-mt| \leq C |t|^{q-1}$.
\end{itemize}
Then there exists a nontrivial solution $u \in H^s(\mathbb{R}^N)$ of
equation (\ref{eq:1}), and this solution is radially symmetric.
\end{theorem}
In the second half of the paper we will show that a direct minimization over the constraint given by the Pohozaev identity need not produce a solution of (\ref{eq:1}). Let us describe what we mean.

For the local laplacian, when the nonlinearity $g$ satisfies condition (\ref{eq:23}), a powerful tool for solving (\ref{eq:1}) is the
\emph{Nehari manifold} $\mathcal{N}$ associated to the functional
$I$. Since $\mathcal{N}$ turns out to be a natural constraint for $I$,
one is led to look for a solution $\bar{u}$ of the minimum problem 
\[
I(\bar{u}) = \min_{u \in \mathcal{N}}I(u).
\]
For example, the assumption that
\[
\sup_{y \in \mathbb{R}^N} V(y) \leq \lim_{|x| \to +\infty} V(x)
\]
guarantees that such a function $\bar{u}$ exists.

However, for a general nonlinearity $g$, this technique no longer
works. It is tempting, therefore, to replace the Nehari manifold
$\mathcal{N}$ with the Pohozaev manifold. 
Since we will prove the following Pohozaev identity
\begin{multline} \label{eq:3}
\frac{N-2s}{2}\int_{\mathbb{R}^N} |\xi|^{2s}|\hat{u}(\xi)|^2 \, d\xi + \frac{N}{2}\int_{\mathbb{R}^N} V(x)|u(x)|^2\, dx + \frac{1}{2} \int_{\mathbb{R}^N} \langle \nabla V(x),x \rangle |u(x)|^2 \, dx \\
= N \int_{\mathbb{R}^N} G(u(x))\, dx,
\end{multline}
we set
\begin{equation*}
  \mathcal{P} = \left\{ u \in H^s(\mathbb{R}^N)\setminus \{0\} \mid \text{$u$ satisfies (\ref{eq:3})} \right\}.
\end{equation*}
Here is our main result about the \emph{non-criticality} of the Pohozaev set. This result was proved in \cite{MR2542090} when $s=1$.
\begin{theorem} \label{th:non}
If we assume (g1--4), (V1), (V3) and
\begin{itemize}
\item[(V5)] $\langle \nabla V(x),x \rangle \leq 0$ for every $x \in \mathbb{R}^N$;
\item[(V6)] $NV(x)+\langle \nabla V(x),x \rangle \geq 0$ for every $x \in \mathbb{R}^N$ and the inequality is strict at some point,
\end{itemize}
then
\[
b = \inf_{u \in \mathcal{P}}I(u)
\]
is not a critical value for the functional $I$.
\end{theorem}

\section{The Pohozaev identity}

To solve (\ref{eq:1}), we will look for critical points of the
functional $I$. In this section we prove that any solution $u \in
H^s(\mathbb{R}^N)$ of (\ref{eq:1}) must satisfy a variational identity
``\`a la Pohozaev''. The following result in sketched in some papers
(\cite{Dipierro2012, Ros-Oton2012}), but its proof is a mixture of
many ingredients that are scattered through the literature.
\begin{proposition}\label{prop:pohozaev-identity}
  Assume that $u \in H^s(\mathbb{R}^N)$ is a (weak) solution to
  (\ref{eq:1}). Then $u$ verifies the Pohozaev identity (\ref{eq:3}).
\end{proposition}
\begin{proof}
  Our argument is borrowed
  from \cite{Frank2010}, where the identity is proved in dimension
  one.
Assume that $u$ satisfies the equation
\begin{equation}
\slap u +V(x)u= g(u) \qquad\text{in $\mathbb{R}^N$}.
\end{equation}
When $s=1$, the standard strategy to prove the Pohozaev identity is to
multiply this equation by $\langle x, \nabla u\rangle$ and integrate
by parts. We will show that this technique works also for the
fractional laplacian, but we need to be more careful, since the
gradient of $u$ need not be integrable, in principle.

\noindent\textbf{Step 1: regularity and decay estimates.} We claim
that $u \in H^1(\mathbb{R}^N)$. Indeed, $u$ belongs to every $L^p$
space by an easy modification of the iteration method in \cite[Proposition 5.1] {Barrios} (or,
equivalently, by the results of \cite{FQT}); moreover $u$ is bounded and $u(x) \to 0$ as $|x|
\to +\infty$. From \cite[Remark
2.11]{Ros-Oton2012a} and recalling that $g$ is a continuous function,
it follows that also $(-\Delta)^{\frac{s}{2}}u \in L^p(\mathbb{R}^N)$
for every finite $p$. Thus $u \in W^{s,p}(\mathbb{R}^N)$ for all finite $p$. Lemma 4.4 of \cite{Cabre2010} guarantees now that $u \in C^{2,\beta}$ for a suitable $\beta \in (0,1)$.
In particular, the gradient of $u$ makes sense.

Finally, we claim that, for some constant $C>0$ and every $x \in \mathbb{R}^N$,
\begin{equation}\label{eq:5}
|u(x)| + |\langle x,\nabla u(x)\rangle| \leq \frac{C}{1+|x|^{N+2s}}.
\end{equation}
Indeed, we recall from Proposition \ref{prop:Kappa} in the appendix that the fundamental
solution $\mathcal{K}$ of the operator $\slap + I$ satisfies the estimates (\ref{eq:14}) and (\ref{eq:15}).
If we write (\ref{eq:1}) as
\begin{equation} \label{eq:9}
u = \mathcal{K}*(-Vu+u+g(u)),
\end{equation}
by exploting the decay of $\mathcal{K}$, the estimate for $u$ is
proved in \cite{FQT}. The decay of the term $|\langle x , \nabla
u\rangle|$ is somehow hidden in the same paper, and follows from the
estimate for $u$ and the estimate for $|\nabla \mathcal{K}|$ by differentiating (\ref{eq:9}). A rather
similar approach is outlined on pages 24--26 of \cite{MR1106251}. Actually, more is true. Indeed, we can prove that $u \in H^{2s+1}(\mathbb{R}^N)$. This follows easily from the decay of $\nabla \mathcal{K}$or, alternatively, by mimicking the proof of Lemma B.1 in \cite{Frank2010} for (\ref{eq:9}).

\noindent\textbf{Step 2: the variational identity.} It is now legitimate to
multiply (\ref{eq:5}) by $\langle x,\nabla u\rangle$, which decays
sufficiently fast at infinity by Step 1. Let us show the computations
for the term containing the fractional laplacian, since all the other
terms are local and can be treated as in the case $s=1$. Recalling the
pointwise identity
\begin{equation*}
\slap \langle x,\nabla u \rangle = 2s \slap u + \langle x,\nabla \slap u\rangle,
\end{equation*}
we can write
\begin{multline*}
\int_{\mathbb{R}^N} \langle x,\nabla u \rangle \slap u \, dx = \int_{\mathbb{R}^N} u \slap  \langle x,\nabla u\rangle \, dx \\
=\int_{\mathbb{R}^N} 2s u \slap u \,dx + \int_{\mathbb{R}^N} u \langle x,\nabla \slap u\rangle \, dx.
\end{multline*}
Now,
\begin{multline*}
\int_{\mathbb{R}^N} \langle x,\nabla \slap u \rangle u \, dx = \int_{\mathbb{R}^N} \operatorname{div} \left( \slap u \cdot u x \right) \, dx - \int_{\mathbb{R}^N} \slap u \operatorname{div}(ux) \, dx \\
= \int_{\mathbb{R}^N} \operatorname{div} \left( \slap u \cdot u x \right) \, dx - \int_{\mathbb{R}^N} \slap u \left( Nu + \langle x,\nabla u \rangle \right)\, dx.
\end{multline*}
Therefore
\begin{multline*}
\int_{\mathbb{R}^N} \langle x,\nabla u \rangle \slap u \, dx = (2s-N) \int_{\mathbb{R}^N} u \slap u \, dx \\
{}+\int_{\mathbb{R}^N} \operatorname{div} \left( \slap u \cdot ux \right) \, dx - \int_{\mathbb{R}^N} \slap u \langle x,\nabla u \rangle \, dx,
\end{multline*}
and then
\begin{equation*}
\int_{\mathbb{R}^N} \langle x,\nabla u \rangle \slap u \, dx = \frac{2s-N}{2}\int_{\mathbb{R}^N} u \slap u \, dx + \frac{1}{2} \int_{\mathbb{R}^N} \operatorname{div} \left( \slap u \cdot ux \right)\, dx.
\end{equation*}
Since $\slap u = g(u)-u$, if we recall the decay estimates of Step 1
and we integrate by parts, we find that the last integral is zero. We conclude that
\begin{equation*}
\int_{\mathbb{R}^N} \langle x,\nabla u \rangle \slap u \, dx = \frac{2s-N}{2} \int_{\mathbb{R}^N} u \slap u \, dx.
\end{equation*}
Since
\begin{equation*}
\int_{\mathbb{R}^N} u \slap u \, dx = \int_{\mathbb{R}^N} |(-\Delta)^{\frac{s}{2}}u|^2 \, dx = \int_{\mathbb{R}^N} |\xi|^{2s} |\hat{u}(\xi)|^2 \, d\xi,
\end{equation*}
the Pohozaev identity (\ref{eq:3}) follows.
\end{proof}

\section{Existence theory}

In this section we want to prove the existence of a radially symmetric
solution to equation (\ref{eq:1}). As usual when dealing with general nonlinearities, we modify the nonlinear term $g$ in a convenient way. Let us distinguish two cases, recalling that $\xi$ is defined in assumption (g4):
\begin{enumerate}
\item if $g(t)>0$ for every $t \geq \xi$, we simply extend $g$ to the negative axis:
\begin{equation*}
\tilde{g}(t)= \left\{
\begin{array}{ll}
g(t) &\text{if $t \geq 0$} \\
-g(-t) &\text{if $t<0$}.
\end{array}
\right.
\end{equation*}
\item If $g$ vanishes somewhere in $[\xi,+\infty)$, we call
\[
t_0 = \min \{ t \geq \zeta \mid g(t)=0 \}
\]
and we define
\begin{equation*}
\tilde{g}(t)= \left\{
\begin{array}{ll}
g(t) &\text{if $t \in [0,t_0]$} \\
0 &\text{if $t \notin [0,t_0]$}\\
-\tilde{g}(-t) &\text{if $t<0$}.
\end{array}
\right.
\end{equation*}
\end{enumerate}
By the maximum principle for the fractional laplacian (see \cite{Silvestre}), any solution of 
\[
(-\Delta)^su+V(x)u=\tilde{g}(u)
\]
is also a solution to (\ref{eq:1}). Therefore, from this moment, we will tacitly write $g$ instead of $\tilde{g}$.
%
We then introduce
\begin{align*}
g_1(t) &= \max \{ g(t)+mt,0\} \\
g_2(t) &= g_1(t)-g(t),
\end{align*}
where $m$ is taken from assumption (g2). It is a simple task to show that
\begin{equation}
\lim_{t \to 0} \frac{g_1(t)}{t}=0
\end{equation}
and
\begin{equation}
\lim_{t \to +\infty} \frac{g_1(t)}{t^{2^\star -1}}=0.
\end{equation}
From
\begin{equation}
g_2(t) \geq mt \quad \text{for all $t \geq 0$}
\end{equation}
it follows that, given any $\ge>0$, there exists $C_\ge>0$ with the property that
\begin{equation}
g_1(t) \leq C_\ge t^{2^\star -1}+\ge g_2(t) \quad\text{for all $t \geq 0$}.
\end{equation}
We now define, for $i=1$, $2$
\begin{equation*}
G_i(t) = \int_0^t g_i(s)\, ds.
\end{equation*}
In particular,
\begin{equation} \label{eq:7}
G_2(t) \geq \frac{m}{2}t^2 \quad\text{for all $t\in\mathbb{R}$}
\end{equation}
and for any $\ge >0$ there exists a number $C_\ge>0$ such that
\begin{equation} \label{eq:8}
G_1(t) \leq \frac{C_\ge}{2^\star} |t|^{2^\star} + \ge G_2(t) \quad\text{for all $t \in \mathbb{R}$}.
\end{equation}
To construct a solution of (\ref{eq:1}) we introduce a parametrized
family of functionals
\begin{multline*}
I_\lambda (u) = \frac{1}{2} \int_{\mathbb{R}^N} |\xi|^{2s}|\hat{u}(\xi)|^2 \, d\xi + \frac{1}{2} \int_{\mathbb{R}^N} V(x)|u(x)|^2 \, dx \\
{}+ \int_{\mathbb{R}^N} G_2(u(x))\, dx - \lambda \int_{\mathbb{R}^N} G_1(u(x))\, dx.
\end{multline*}
Since $I_1=I$, we will construct bounded Palais-Smale sequences for
almost every $\lambda$ close to 1, and the exploit the following
theorem. It is a simple modification of \cite[Theorem 1.1]{MR1718530},
stated by \cite{MR2542090}.
\begin{theorem} \label{thm:jeanjean}
Let $X$ be a Banach space and let $J \subset [0,+\infty)$ an interval. Consider the family of functionals on $X$ given by
\[
I_\lambda (u) = A(u)-\lambda B(u),
\]
where $\lambda \in J$. Assume that $B$ is nonnegative and either $A(u)
\to +\infty$ or $B(u) \to +\infty$ as $\|u\| \to +\infty$. Moreover, assume that $I_\lambda (0)=0$ for every $\lambda \in J$.

For $j \in J$ we set
\begin{equation}
\Gamma_\lambda = \left\{ \gamma \in C([0,1],X) \mid \gamma (0)=0, \ I_\lambda (\gamma(1))<0 \right\}.
\end{equation}
If, for every $\lambda \in J$, $\Gamma_\lambda \neq \emptyset$  and
\begin{equation}\label{eq:6}
c_\lambda = \inf_{\gamma \in \Gamma_\lambda} \max_{t \in [0,1]} I_\lambda (\gamma (t))>0,
\end{equation}
then for almost every $\lambda \in J$ there exists a sequence $\{v_n\}_n \subset X$ such that
\begin{enumerate}
\item $\{v_n\}_n$ is bounded;
\item $I_\lambda (v_n) \to c_\lambda$;
\item $DI_\lambda (v_n) \to 0$ strongly in $X^*$.
\end{enumerate}
\end{theorem}
We want to use this result with
\begin{align*}
X &= H^s_{\mathrm{rad}} = \left\{ u \in H^s(\mathbb{R}^N) \mid \hbox{$u$ is radially symmetric} \right\} \\
A(u) &=  \frac{1}{2} \int_{\mathbb{R}^N} |\xi|^{2s}|\hat{u}(\xi)|^2 \, d\xi + \frac{1}{2} \int_{\mathbb{R}^N} V(x)|u(x)|^2 \, dx + \int_{\mathbb{R}^N} G_2(u(x))\, dx \\
B(u) &= \int_{\mathbb{R}^N} G_1(u(x))\, dx.
\end{align*}
The rest of this section is devoted to the definition of an interval
$J$ such that $\Gamma_\lambda \neq \emptyset$ and (\ref{eq:6}) holds true for
every $\lambda \in J$.

To begin with, we recall the following result from \cite{Dipierro2012}:
\begin{lemma}
Let $\mathfrak{z}$ and $R$ be two positive numbers. Define
\[
v_R(t)= \left\{
\begin{array}{ll}
\mathfrak{z} &\text{if $t \in [0,R]$} \\
\mathfrak{z}(R+1-t) &\text{if $t \in (R,R+1)$}\\
0 &\text{if $t \in [R+1,+\infty)$}.
\end{array}
\right.
\]
Finally, set $w_R(x) = v_R(|x|)$. Then $w_R \in H^s(\mathbb{R}^N)$ and $\|w_R\|_{H^s} \leq C(N,s,R)\mathfrak{z}$ for some constant $C(N,s,R)>0$. 

Moreover, there exists $R>0$ such that 
\[
\int_{\mathbb{R}^N} G_1(w_R(x))\, dx - \int_{\mathbb{R}^N} G_2(w_R(x))\, dx = \int_{\mathbb{R}^N} G(w_R(x))\, dx >0.
\]
\end{lemma}
If $R>0$ is the number given by the previous Lemma, we keep it fixed
and abbreviate $z = w_R$. We define
\begin{equation}
J = \left[ \bar{\delta},1 \right],
\end{equation}
where $0<\bar{\delta}<1$ is chosen so that
\[
\bar{\delta}\int_{\mathbb{R}^N} G_1(w_R(x))\, dx - \int_{\mathbb{R}^N}
G_2(w_R(x))\, dx >0.
\]
\begin{lemma}
\begin{itemize}
\item[(a)] For every $\lambda \in J$, the set $\Gamma_\lambda$ is non-empty.
\item[(b)] $\inf_{\lambda \in J} c_\lambda >0$.
\end{itemize}
\end{lemma}
\begin{proof}
  Fir any $\lambda \in J$. To prove (a), consider a large number
  $\bar{\theta}>0$ and set $\bar{z}=z(\cdot / \bar{\theta})$. We can define the following path in $H^s_{\mathrm{rad}}$:
\begin{equation*}
\gamma(t)= \left\{
\begin{array}{ll}
0 &\text{if $t=0$}\\
\bar{z}^t = \bar{z}(\cdot/t) &\text{if $0<t \leq 1$}.
\end{array}
\right.
\end{equation*}
Since
\begin{multline*}
I_\lambda(\gamma(1)) \leq \frac{\bar{\theta}^{N-2s}}{2} \int_{\mathbb{R}^N} |\xi|^{2s} |\hat{u}(\xi)|^2 \, d\xi + \frac{\bar{\theta}^N}{2} \int_{\mathbb{R}^N} V(\bar{\theta}x) |z(x)|^2\, dx \\
{}+\bar{\theta}^N \left( \int_{\mathbb{R}^N}G_2(z(x))\, dx - \bar{\delta} \int_{\mathbb{R}^N} G_1(z(x))\, dx \right),
\end{multline*}
we can take $\bar{\theta}$ so large that $I_\lambda(\gamma(1))<0$.

To prove (b), we use (\ref{eq:7}) and (\ref{eq:8}) and remark that these imply
\begin{multline*}
I_\lambda (u) \geq \frac{1}{2} \int_{\mathbb{R}^N} |\xi|^{2s} |\hat{u}(\xi)|^2 \, d\xi + \frac{1}{2}\int_{\mathbb{R}^N} V(x) |u(x)|^2\, dx \\
{} \quad + \int_{\mathbb{R}^N} G_2(u(x))\, dx - \int_{\mathbb{R}^N} G_1(u(x))\, dx \\
\geq \frac{1}{2} \int_{\mathbb{R}^N} |\xi|^{2s} |\hat{u}(\xi)|^2 \, d\xi + \left(1-\ge \right) \frac{m}{2} \int_{\mathbb{R}^N} |u(x)|^2\, dx - \frac{C_\ge}{2^\star} \int_{\mathbb{R}^N} |u(x)|^{2^\star}\, dx.
\end{multline*}
Recalling the Sobolev embedding $H^s \subset L^{2^\star}$, we conclude that, for some $\rho>0$, $\|u\|_{H^s} \leq \rho$ implies $I_\lambda(u)>0$. Let
\[
\tilde{c} = \inf_{\|u\|=\rho} I_\lambda (u)>0.
\]
If $\lambda \in J$ and $\gamma \in \Gamma_\lambda$, certainly
$\|\gamma(1)\|>\rho$. Since $\gamma$ is continuous, there is $t_\gamma
\in (0,1)$ such that $\|\gamma(t_\gamma)\|=\rho$. Hence
\[
c_\lambda \geq \inf_{\gamma \in \Gamma_\lambda} I_\lambda (\gamma
(t_\gamma)) \geq \tilde{c}
\]
and the proof is complete.
\end{proof}
The next step is the verification of the Palais-Smale condition for $I_\lambda$.
\begin{lemma} \label{lem:54}
  For every $\lambda\in J$ and $1/2<s<1$, the functional $I_\lambda$ satisfies the \emph{bounded} Palais-Smale condition: from every bounded Palais-Smale sequence it is possible to extract a converging subsequence.
\end{lemma}
\begin{proof}
  Pick $\lambda \in J$, and assume $\{u_n\}_n$ is a sequence in
  $H_{\mathrm{rad}}^s$ such that
\begin{align*}
  &\left| I_\lambda (u_n) \right| \leq C \\
  &DI_\lambda (u_n) \to 0 \quad\text{strongly in the dual space 
    $\left(H^s_{\mathrm{rad}}\right)^*$}.
\end{align*}
Up to subsequences, we may assume also that $u_n \to u$ almost
everywhere and weakly in $H_{\mathrm{rad}}^s$. Hence
\begin{equation*}
\int_{\mathbb{R}^N} |\xi|^{2s} |\hat{u}(\xi)|^2 \, d\xi \leq \liminf_{n \to +\infty} \int_{\mathbb{R}^N} |\xi|^{2s} |\hat{u}_n(\xi)|^2 \, d\xi
\end{equation*}
and
\begin{equation*}
\int_{\mathbb{R}^N} V(x) |u(x)|^2\, dx \leq \liminf_{n \to +\infty} \int_{\mathbb{R}^N} V(x) |u_n(x)|^2 \, dx.
\end{equation*}
Applying the first part of Strauss' compactness lemma \ref{lem:strauss}, we conclude that 
\[
\lim_{n \to +\infty} \int_{\mathbb{R}^N} g_i(u_n(x))h(x)\, dx = \int_{\mathbb{R}^N}g_i(u(x))h(x)\, dx
\]
for every $h \in C_0^\infty (\mathbb{R}^N)$, and therefore $DI_\lambda (u)=0$. As a consequence,
\begin{multline*}
\int_{\mathbb{R}^N} |\xi|^{2s}|\hat{u}(\xi)|^2 \, d\xi + \int_{\mathbb{R}^N} V(x)|u(x)|^2\, dx \\
=\int_{\mathbb{R}^N} \left( \lambda g_1(u(x))u(x)-g_2(u(x))u(x) \right) dx
\end{multline*}
by the Pohozaev identity. Again by Lemma \ref{lem:strauss} and Lemma
\ref{th:decay} and recalling that $1/2<s<1$,
\begin{equation} \label{eq:18}
\lim_{n \to +\infty}\int_{\mathbb{R}^N}g_1(u_n(x))u_n(x)\, dx = \int_{\mathbb{R}^N} g_1(u(x))u(x)\, dx
\end{equation}
and
\begin{equation*}
\int_{\mathbb{R}^N} g_2(u(x))u(x)\, dx \leq \liminf_{n \to +\infty} \int_{\mathbb{R}^N} g_2(u_n(x))u_n(x)\, dx.
\end{equation*}
We deduce now that
\begin{multline*}
\limsup_{n \to +\infty} \int_{\mathbb{R}^N} |\xi|^{2s}|\hat{u}_n(\xi)|^2\, d\xi + \int_{\mathbb{R}^N} V(x)|u_n(x)|^2\, dx = \\
\limsup_{n \to +\infty} \int_{\mathbb{R}^N} \left( \lambda g_1(u_n(x))u_n(x)-g_2(u_n(x))u_n(x) \right) dx \\
\leq \lambda \int_{\mathbb{R}^N} g_1(u(x))u(x)\, dx - \int_{\mathbb{R}^N}g_2(u(x))u(x)\, dx \\
=\int_{\mathbb{R}^N} |\xi|^{2s}|\hat{u}(\xi)|^2 \, d\xi + \int_{\mathbb{R}^N}V(x)|u(x)|^2\, dx.
\end{multline*}
This means that
\begin{align*}
\lim_{n \to +\infty}  \int_{\mathbb{R}^N} |\xi|^{2s}|\hat{u}_n(\xi)|^2\, d\xi &=  \int_{\mathbb{R}^N} |\xi|^{2s}|\hat{u}(\xi)|^2\, d\xi \\
\lim_{n \to +\infty} \int_{\mathbb{R}^N} V(x)|u_n(x)|^2\, dx &= \int_{\mathbb{R}^N} V(x)|u(x)|^2\, dx,
\end{align*}
and finally
\begin{equation} \label{eq:21}
\lim_{n \to +\infty} \int_{\mathbb{R}^N} g_2(u_n(x))u_n(x)\, dx = \int_{\mathbb{R}^N} g_2(u(x))u(x)\, dx.
\end{equation}
Since we can write $g_2(s)s = ms^2+q(s)$ for some non-negative,
continuous function $q$, we conclude that $u_n \to u$ strongly in
$L^2(\mathbb{R}^N)$ and in $H^s_{\mathrm{rad}}$. Indeed, Fatou's lemma yields
\begin{equation}\label{eq:19}
\int_{\mathbb{R}^N} |u(x)|^2\, dx \leq \liminf_{n \to +\infty} \int_{\mathbb{R}^N} |u_n(x)|^2 \, dx 
\end{equation}
and
\begin{equation}\label{eq:20}
\int_{\mathbb{R}^N} q(u(x))\, dx \leq \liminf_{n \to +\infty} \int_{\mathbb{R}^N} q(u_n(x))\, dx.
\end{equation}
Therefore, by (\ref{eq:21}),
\begin{equation*}
\int_{\mathbb{R}^N} m|u_n(x)|^2 \, dx = \int_{\mathbb{R}^N}m|u(x)|^2 \, dx + \int_{\mathbb{R}^N} q(u(x))\, dx - \int_{\mathbb{R}^N}q(u_n(x))\, dx + o(1)
\end{equation*}
and by (\ref{eq:20})
\begin{multline*}
\limsup_{n \to +\infty} \int_{\mathbb{R}^N} m|u_n(x)|^2 \, dx \\
\leq \int_{\mathbb{R}^N} m|u(x)|^2 \, dx + \limsup_{n \to +\infty} \left( \int_{\mathbb{R}^N} q(u(x))\, dx - \int_{\mathbb{R}^N} q(u_n(x))\, dx \right) \\
\leq \int_{\mathbb{R}^N} m|u(x)|^2 \, dx  + \int_{\mathbb{R}^N} q(u(x))\, dx - \liminf_{n \to +\infty} \int_{\mathbb{R}^N} q(u_n(x))\, dx \leq \int_{\mathbb{R}^N} m|u(x)|^2 \, dx.
\end{multline*}
This and (\ref{eq:19}) imply that $u_n \to u$ in $L^2(\mathbb{R}^N)$. and hence in $H^s_{\mathrm{rad}}$.
\end{proof}
If we apply the previous lemmas and Theorem \ref{thm:jeanjean}, we
reach the following conclusion.
\begin{proposition}
  For every $s \in (1/2,1)$ and almost every $\lambda \in J$, there exists $u^\lambda \in
  W_{\mathrm{rad}}^{s,2}(\mathbb{R}^N)$ such that $u^\lambda \neq 0$,
  $I_\lambda(u^\lambda) = c_\lambda$, and $DI_\lambda(u^\lambda)=0$.
\end{proposition}
\subsection{The proof of Theorem \ref{thm:existence}}

We select a sequence $\{\lambda_n\}_n$ of numbers
$\lambda_n \uparrow 1$ such that for each $n \in \mathbb{N}$ there
exists $v_n \in H_{\mathrm{rad}}^{s}(\mathbb{R}^N)$ with $v_n \neq
0$ and
\begin{align*}
  I_{\lambda_n}(v_n) &= c_{\lambda_n} \\
  DI_{\lambda_n}(v_n) &= o(1) \quad\text{strongly in $\left(
      H_{\mathrm{rad}}^{s}(\mathbb{R}^N) \right)^*$}.
\end{align*}
Each $v_n$ is a solution of the equation
\begin{equation*}
\slap v_n + Vv_n + g_2(v_n)-\lambda_n g_1 (v_n)=0,
\end{equation*}
and therefore
\begin{multline} \label{eq:10}
\frac{N-2s}{2}\int_{\mathbb{R}^N} |\xi|^{2s}|\hat{u}(\xi)|^2 \, d\xi + \frac{1}{2}\int_{\mathbb{R}^N} \langle \nabla V(x),x\rangle |v_n(x)|^2\, dx \\
{}+\frac{N}{2}\int_{\mathbb{R}^N} V(x)|v_n(x)|^2 \, dx + N \int_{\mathbb{R}^N} \left(G_2(v_n(x))-\lambda_n G_1(v_n(x)) \right)dx=0.
\end{multline}
If we set, for $i=1$, $2$,
\begin{align*}
\alpha_n &= \int_{\mathbb{R}^N} |\xi|^{2s}|\hat{u}(\xi)|^2\, d\xi \\
\beta_n &= \int_{\mathbb{R}^N} V(x)|v_n(x)|^2 \, dx \\
\eta_n &= \int_{\mathbb{R}^N} \langle \nabla V(x),x\rangle |v_n(x)|^2\, dx\\
\gamma_{i,n} &= \int_{\mathbb{R}^N} G_i(v_n(x))\, dx \\
\delta_{i,n} &= \int_{\mathbb{R}^N} g_i(v_n(x))v_n(x)\, dx
\end{align*}
we deduce from (\ref{eq:10}) that
\begin{equation} \label{eq:11}
\left\{
\begin{array}{l}
\frac{\alpha_n+\beta_n}{2}+\gamma_{2,n}-\lambda_n \gamma_{1,n}=c_{\lambda_n} \\
\alpha_n+\beta_n+\delta_{2,n}-\lambda_n \delta_{1,n} = 0 \\
\alpha_n + \frac{N}{N-2s}\beta_n + \frac{\eta_n}{N-2s}+\frac{2N}{N-2s}\gamma_{2,n}-\frac{2N}{N-2s}\lambda_N \gamma_{1,n}=0.
\end{array}
\right.
\end{equation}
Some algebraic manipulations imply easily that
\begin{equation*}
\left( \frac{N}{N-2s}-1 \right) \alpha_n - \frac{\eta_n}{N-2s} = \frac{2N}{N-2s}c_{\lambda_n},
\end{equation*}
i.e.
\begin{equation*}
\frac{s}{N}\alpha_n - \frac{\eta_n}{2N}=c_{\lambda_n},
\end{equation*}
and it follows that $\{\alpha_n\}_n$ is bounded from above. From the second equation in (\ref{eq:11}) it follows that 
\begin{equation*}
\delta_{2,n}-\lambda_n \delta_{1,n}=-\alpha_n - \beta_n \leq 0
\end{equation*}
and there exist $\ge>0$ and $C_\ge>0$ such that
\begin{equation*}
\delta_{2,n} \leq \delta_{1,n} \leq C_\ge \int_{\mathbb{R}^N} |v_n(x)|^{2^\star}\, dx + \ge \delta_{2,n}.
\end{equation*}
As a consequence,
\begin{equation*}
\left(1-\ge \right) \delta_{2,n} \leq C_\ge \int_{\mathbb{R}^N} |v_n(x)|^{2^\star}\, dx
\end{equation*}
and $\{\delta_{2,n}\}_n$ is also bounded from above. Finally, this
implies that $\{v_n\}_n$ is bounded in
$H_{\mathrm{rad}}^{s}(\mathbb{R}^N)$, and we may assume that $v_n
\rightharpoonup v$ weakly in $H_{\mathrm{rad}}^{s}(\mathbb{R}^N)$. Since $\{g_1(v_n)\}_n$ is bounded in $\left(H^s_{\mathrm{rad}}(\mathbb{R}^N) \right)^*$ by Lemma \ref{lem:strauss} and
\[
\int_{\mathbb{R}^N}g_1(v(x))h(x)\, dx =
\int_{\mathbb{R}^N}g_1(v_n(x))h(x)\, dx + o(1)
\]
for every $h \in C_0^\infty(\mathbb{R}^N)$, we deduce that
\begin{equation*}
DI(v_n) = DI_{\lambda_n}(v_n)+\left( \lambda_n -1 \right) g_1(v_n) = \left( \lambda_n -1 \right) g_1(v_n)=o(1).
\end{equation*}
Moreover,
\begin{equation*}
I(v_n) = I_{\lambda_n}(v_n) + \left( \lambda_n-1 \right) \int_{\mathbb{R}^N}G_1(v_n(x))\, dx = c + o(1).
\end{equation*}
Hence $\{v_n\}_n$ is a Palais-Smale sequence for $I$ at level $c$, and
we conclude that $v$ is a non-trivial solution of the equation
$DI(v)=0$. This completes the proof of Theorem \ref{thm:existence}.

\subsection{The proof of Theorem \ref{thm:existence2}}

The proof of Theorem \ref{thm:existence2} is similar to that of Theorem \ref{thm:existence}. The main
difficulty is that, in Lemma \ref{lem:54}, we cannot use Lemma
\ref{lem:strauss} and the pointwise decay of $u_n$ to prove (\ref{eq:18}). However, Theorem \ref{th:compact} tells us that $\{u_n\}_n$ is
relatively compact in $L^q(\mathbb{R}^N)$, $2<q<2^\star$. Inserting
this information into assumption (g3)', we conclude that
$\{g_1(u_n)u_n\}_n$ converges strongly to $g_1(u)u$. The proof is then
identical to that of Theorem \ref{thm:existence}.
\begin{remark}
The convergence $g_1(u_n)u_n \to g_1(u)u$ was troublesome because the assumptions on $g$ are rather weak. The philosophy behind the use of radially symmetric functions is that they rule out any \emph{mass displacement} to infinity: this is precisely the content of Strauss' decay lemma. The fact that $g(s)=o(s^{2^\star-1})$ as $s \to +\infty$ is a much weaker condition than a pure subcritical growth, and does not imply the continuity of the superposition operator $u \mapsto g_1(u)u$.
\end{remark}

\section{Non-critical values}

As we said in a previous section, the idea of minimizing the Euler functional $I$ on the set of those functions that satisfy the Pohozaev identity (\ref{eq:3}) can be seen as a natural attempt to find ground-state solutions to (\ref{eq:1}). However, the potential function $V$ can be an obstruction, as we shall see.

\begin{proposition}\label{lem:6.2}
Let us define
\begin{equation}
\mathcal{P} = \left\{ u \in H^s(\mathbb{R}^N)\setminus \{0\} \mid \text{$u$ satisfies (\ref{eq:3})} \right\}.
\end{equation}
The following facts hold true.
\begin{enumerate}
\item There results
\begin{equation*}
\inf \left\{ \int_{\mathbb{R}^N} |\xi|^{2s} |\hat{u}(\xi)|^2 \, d\xi \mid u \in \mathcal{P} \right\} >0.
\end{equation*}
\item There results
\begin{equation*}
b = \inf_{u \in \mathcal{P}} I(u)>0.
\end{equation*}
\item Let $w \in H^s(\mathbb{R}^N)$ be such that $\int_{\mathbb{R}^N}G(w(x))\, dx >0$. Then there exists $\bar{\theta}>0$ such that $w^{\bar{\theta}}  = w(\cdot / \bar{\theta}) \in \mathcal{P}$.
\end{enumerate}
\end{proposition}
\begin{proof}
\begin{enumerate}
\item The proof is standard, and follows from (\ref{eq:3}) and assumption (V6).
\item Indeed, if $u \in \mathcal{P}$, then
\begin{equation} \label{eq:17}
I(u)=\frac{s}{N} \int_{\mathbb{R}^N} |\xi|^{2s} |\hat{u}(\xi)|^2 \, d\xi - \frac{1}{2N} \int_{\mathbb{R}^N} \langle \nabla V(x),x \rangle |u(x)|^2 \, dx,
\end{equation}
and the assertion follows from the previous Lemma, assumption (g1) and assumption (V2).
\item We notice that
\begin{multline*}
I(w^{\bar{\theta}}) = \frac{\bar{\theta}^{N-2s}}{2} \int_{\mathbb{R}^N} |\xi|^{2s} |\hat{w}(\xi)|^2 \, d\xi \\
{}+ \frac{\bar{\theta}^N}{2} \int_{\mathbb{R}^N} V(\bar{\theta}x) |w(x)|^2 \, dx - \bar{\theta}^N \int_{\mathbb{R}^N} G(w(x))\, dx.
\end{multline*}
First of all, we remark that $I(w^{\bar{\theta}})>0$ when $\bar{\theta}$ is sufficiently small.
Since our assumptions on $V$ imply immediately that
\begin{equation*}
\lim_{\bar{\theta} \to +\infty} \int_{\mathbb{R}^N} V(\bar{\theta}x) |w(x)|^2 \, dx =0,
\end{equation*}
we conclude that $\lim_{\bar{\theta} \to +\infty} I(w^{\bar{\theta}}) = -\infty$. Hence the function $\bar{\theta} \mapsto I(w^{\bar{\theta}})$ must have at least a critical point. For this particular $\bar{\theta}>0$, we have $w^{\bar{\theta}}\in \mathcal{P}$.
\end{enumerate}
\end{proof}
We define now
\begin{equation*}
  \mathcal{P}_0 = \left\{ u \in H^s(\mathbb{R}^N) \setminus \{0\} \mid \frac{N-2s}{2}\int_{\mathbb{R}^N} |\xi|^{2s}|\hat{u}(\xi)|^2 \, d\xi = N \int_{\mathbb{R}^N}G(u(x))\, dx \right\}.
\end{equation*}
This set is defined exactly by the Pohozaev
identity for solutions $u \in H^s(\mathbb{R}^N)$ of the equation
\begin{equation} \label{eq:2}
\slap u = g(u) \qquad\text{in $\mathbb{R}^N$}.
\end{equation}
It can be easily checked that $\mathcal{P}_0$ is a natural constraint for the Euler functional
\begin{equation} \label{eq:22}
  I_0(u)=\frac{1}{2} \int_{\mathbb{R}^N} |\xi|^{2s}|\hat{u}(\xi)|^2\, d\xi - \int_{\mathbb{R}^N} G(u(x))\, dx
\end{equation}
and that the celebrated result by Jeanjean and Tanaka (see
\cite{MR1974637}) still holds in our setting, so that $\min_{u \in
  \mathcal{P}_0} I_0(u)$ coincides with the minimum of $I_0(u)$ as $u$
ranges over all the nontrivial solutions of (\ref{eq:2}).

If $w \in \mathcal{P}_0$ and $y \in \mathbb{R}^N$, we set $w_y = w(\cdot -y) \in \mathcal{P}_0$.
Let us fix $\theta_y>0$ such that $\tilde{w}_y = w_y(\cdot/\theta_y) \in \mathcal{P}$.

\begin{lemma}
There results
\begin{equation*}
\lim_{|y| \to +\infty} \theta_y =1.
\end{equation*}
\end{lemma}
\begin{proof}
\textbf{Claim \#1:} $\limsup_{|y| \to +\infty} \theta_y < +\infty$.

If not, $\theta_{y_n} \to +\infty$ along some sequence $\{y\}_n$ with $|y_n| \to +\infty$. Given $y \in \mathbb{R}^N$, we compute
\begin{multline*}
I(\tilde{w}_y) = \frac{\theta_{y}^{N-2s}}{2} \int_{\mathbb{R}^N} |\xi|^{2s} |\hat{w}(\xi)|^2\, d\xi + \frac{\theta_y^N}{2}\int_{\mathbb{R}^N} V(\theta_y x)\left| w \left( x - \frac{y}{\theta_y} \right) \right|^2 \, dx \\
{}- \theta_y^N \int_{\mathbb{R}^N} G(w(x))\, dx.
\end{multline*}
Now,
\begin{multline*}
\int_{\mathbb{R}^N} V(\theta_y x)\left| w \left( x - \frac{y}{\theta_y} \right) \right|^2 \, dx = \\\int_{B(0,\rho)} V(\theta_y x)\left| w \left( x - \frac{y}{\theta_y} \right) \right|^2 \, dx + \int_{\mathbb{R}^N \setminus B(0,\rho)} V(\theta_y x)\left| w \left( x - \frac{y}{\theta_y} \right) \right|^2 \, dx \\
\leq \|V\|_\infty \int_{B(-\frac{y}{\theta_y},\rho)} |w(x)|^2 \, dx + \sup_{x \notin B(0,\rho)} |V(x)| \|w\|_{L^2}^2.
\end{multline*}
Pick $\ge>0$ and choose $\bar{\rho}>0$ such that
\begin{equation*}
\|V\|_\infty \int_{B(-\frac{y}{\theta_y},\rho)} |w(x)|^2 \, dx \leq \ge
\end{equation*}
for any $y \in \mathbb{R}^N$ and any $\rho < \bar{\rho}$. Hence
\begin{equation*}
\lim_{|y| \to +\infty} \int_{\mathbb{R}^N} V(\theta_y x)\left| w \left( x - \frac{y}{\theta_y} \right) \right|^2 \, dx  =0.
\end{equation*}
We deduce that $\lim_{n \to +\infty} I(\tilde{w}_{y_n}) =-\infty$, which is a contradiction to Lemma \ref{lem:6.2}. This proves Claim \#1.

\medskip

\noindent\textbf{Claim \#2:} $\lim_{|y| \to +\infty} \theta_y =1$.

Indeed, since $w \in \mathcal{P}_0$ and $\tilde{w}_y \in \mathcal{P}$,
\begin{multline} \label{eq:30}
N (\theta_y^2-1) \int_{\mathbb{R}^N} G(w(x))\, dx \\
= \frac{1}{2} \theta_y^2 \int_{\mathbb{R}^N} \left(
NV(\theta_y x + y) + \langle \nabla V(\theta_y x+y),\theta_y x+y \rangle 
\right) |w(x)|^2 \, dx
\end{multline}
Recalling our assumptions (V5) and (V6),
\begin{multline*}
0 \leq \int_{\mathbb{R}^N} \left( NV(\theta_y x+y)+\langle \nabla V(\theta_y x+y),\theta_y x+y \rangle \right)|w(x)|^2 \, dx \\
\leq \int_{\mathbb{R}^N} NV(\theta_y x+y) |w(x)|^2 \, dx = o(1)
\end{multline*}
as $|y| \to +\infty$ by Dominated Convergence.
Claim \#1 shows that the right-hand side of (\ref{eq:30}) is $o(1)$ as $|y| \to +\infty$: we conclude that $\theta_y = 1 + o(1)$ as $|y| \to +\infty$.
\end{proof}

\begin{proposition} \label{lem:6.5}
We define
\begin{equation*}
b_0 = \inf \left\{ I_0(u) \mid u \in \mathcal{P}_0 \right\},
\end{equation*}
where $I_0$ was defined in (\ref{eq:22}). The following facts hold true.
\begin{enumerate}
\item There results $b \leq b_0$.
\item Let $z \in H^s(\mathbb{R}^N)$ be such that $\int_{\mathbb{R}^N} G(z(x))\, dx >0$. Then there exists $\bar{\theta}>0$ such that $z^{\bar{\theta}}=z(\cdot / \bar{\theta}) \in \mathcal{P}_0$. In particular, this is true for any $z \in \mathcal{P}$ with $\bar{\theta}\leq 1$.
\end{enumerate}
\end{proposition}
\begin{proof}
\begin{enumerate}
\item Indeed, let $w \in H^s(\mathbb{R}^N)$ be a ground-state solutions of
\begin{equation} \label{eq:13}
\slap w = g(w),
\end{equation}
whose existence is proved in \cite{Dipierro2012}. In particular, $w \in \mathcal{P}_0$ and $I_0(w)=b_0$. Since (\ref{eq:13}) is invariant under translations, $w_y \in \mathcal{P}_0$ and $I_0(w_y) = b_0$ for any $y \in \mathbb{R}^N$.

Let us fix $\theta_y > 0$ such that $\tilde{w}_y \in \mathcal{P}$. Therefore
\begin{multline*}
\left| I(\tilde{w}_y)-b_0 \right| = \left| I(\tilde{w}_y)-I_0(w_y) \right| \leq \\
\frac{|\theta_y^{N-2s}-1|}{2} \int_{\mathbb{R}^N} |\xi|^{2s} |\hat{w}(\xi)|^2 \, d\xi + \frac{\theta_y^N}{2} \int_{\mathbb{R}^N} V(\theta_y x + y) |w(x)|^2\, dx \\
+ |\theta_y^N-1| \int_{\mathbb{R}^N} G(w(x))\, dx.
\end{multline*}
Letting $|y| \to +\infty$, we see that $I(\tilde{w}_y) \to b_0$, and hence $b \leq b_0$.
\item There clearly exists $\bar{\theta}>0$ such that
\begin{equation*}
\frac{N-2s}{2} \int_{\mathbb{R}^N} |\xi|^{2s} |\hat{z}(\xi)|^2 \, d\xi = N \bar{\theta}^2 \int_{\mathbb{R}^N} G(z(x))\, dx.
\end{equation*}
Consider now the case $z \in \mathcal{P}$. Since
\begin{multline*}
\frac{N-2s}{2} \int_{\mathbb{R}^N} |\xi|^{2s} |\hat{z}(\xi)|^2 \, d\xi + \frac{N}{2} \int_{\mathbb{R}^N} V(x) |z(x)|^2 \, dx \\
{}+ \frac{1}{2} \int_{\mathbb{R}^N} \langle \nabla V(x),x \rangle |z(x)|^2 \, dx 
= N \int_{\mathbb{R}^N} G(z(x))\, dx,
\end{multline*}
by (V6) we have $\int_{\mathbb{R}^N} G(z(x))\, dx >0$. If $\bar{\theta}>0$ is chosen so that $z^{\bar{\theta}} \in \mathcal{P}_0$, then
\begin{equation} \label{eq:16}
\frac{1}{2} \int_{\mathbb{R}^N} \left( NV(x)+ \nabla V(x),x \rangle |z(x)|^2 \right)dx = N(1-\bar{\theta}^2) \int_{\mathbb{R}^N} G(z(x)).
\end{equation}
Hence $0<\bar{\theta} \leq 1$.
\end{enumerate}
\end{proof}

\subsection{Proof of Theorem \ref{th:non}}

Assume, by contradiction, the existence of a critical point $z \in H^s(\mathbb{R}^N)$ of $I$ at level $b$; as a consequence, $z \in \mathcal{P}$ and $I(z)=b$. Fix $\theta \in (0,1]$ such that $z^\theta \in \mathcal{P}_0$; by the strong maximum principle (see \cite{Silvestre}), we can assume that $z >0$. By assumption (V6) and (\ref{eq:16}) we conclude that $\theta <1$.

From assumption (V5) and (\ref{eq:17}) we infer that
\begin{multline*}
b = I(z) = \frac{s}{N} \int_{\mathbb{R}^N} |\xi|^{2s} |\hat{z}(\xi)|^2 \, d\xi - \frac{1}{2N} \int_{\mathbb{R}^N} \langle \nabla V(x),x \rangle |z(x)|^2 \, dx \\
> \frac{s\theta^{N-2s}}{N} \int_{\mathbb{R}^N} |\xi|^{2s} |\hat{z}(\xi)|^2 \, d\xi = I_0(z^\theta) \geq b_0.
\end{multline*}
But this contradicts Lemma \ref{lem:6.5}, part 1.

\section{Appendix}

A basic regularity theory for the fractional laplacian is based on the following result.

\begin{proposition}[\cite{FQT}]
Assume $p \geq 1$ and $\beta>0$.
\begin{enumerate}
\item For $s \in (0,1)$ and $2s<\beta$, we have $\slap \colon W^{s,p}(\mathbb{R}^N) \to W^{\beta-2s,p}(\mathbb{R}^N)$.
\item For $s$, $\gamma \in (0,1)$ and $0< \mu \leq \gamma -2s$, we have $\slap \colon C^{0,\gamma}(\mathbb{R}^N) \to C^{0,\mu}(\mathbb{R}^N)$ if $2s<\gamma$, and $\slap \colon C^{1,\gamma}(\mathbb{R}^N) \to C^{1,\mu}(\mathbb{R}^N)$ if $2s>\gamma$.
\end{enumerate}
\end{proposition}

For the reader's convenience, we recall the main properties of the operator $\mathcal{K}=\left( \slap + I \right)^{-1}$. It is known that 
\begin{equation*}
\mathcal{K} = \mathcal{F}^{-1} \left( \frac{1}{1+|\xi|^{2s}} \right).
\end{equation*}

\begin{proposition}[\cite{FQT}] \label{prop:Kappa}
Let $N \geq 2$ and $s \in (0,1)$. Then we have:
\begin{enumerate}
\item $\mathcal{K}$ is positive, radially symmetric and smooth on $\mathbb{R}^N \setminus \{0\}$. Moreover, it is non increasing as a function of $r=|x|$.
\item For appropriate constants $C_1$ and $C_2$, 
\begin{align} 
\mathcal{K}(x) \leq \frac{C_1}{|x|^{N+2s}} \qquad &\text{if $|x| \geq 1$} \label{eq:14}\\
\mathcal{K}(x) \leq \frac{C_2}{|x|^{N-2s}} \qquad &\text{if $|x| \leq 1$} \label{eq:15}
\end{align}
\item There is a constant $C>0$ such that 
\begin{equation}
|\nabla \mathcal{K}(x)| \leq \frac{C}{|x|^{N+1+2s}}, \qquad |D^2\mathcal{K}(x)| \leq \frac{C}{|x|^{N+2+2s}}
\end{equation}
if $|x|\geq 1$.
\item If $q \geq 1$ and $N-2s-\frac{N}{q} < s < N+2s-\frac{N}{q}$, then $|x|^{s}\mathcal{K}(x) \in L^q(\mathbb{R}^N)$.
\item If $1 \leq q < \frac{N}{N-2s}$, then $\mathcal{K} \in L^q(\mathbb{R}^N)$.
\item $|x|^{N+2s}\mathcal{K}(x) \in L^\infty (\mathbb{R}^N)$.
\end{enumerate}
\end{proposition}

We collect here some useful results about compactness and function
spaces. The first is a slight modification of a popular compactness
criterion by Strauss (see \cite{MR0454365} and \cite{MR695535}).
\begin{lemma} \label{lem:strauss}
Let $P$ and $Q$ be two real-valued functions of one real variable such that
\begin{equation*}
\lim_{s \to +\infty} \frac{P(s)}{Q(s)}=0.
\end{equation*}
Let $\{v_n\}_n$, $v$ and $z$ be measurable functions from
  $\mathbb{R}^N$ to $\mathbb{R}$, with $z$ bounded, such that
\begin{align*}
&\sup_n \int_{\mathbb{R}^N} |Q(v_n(x))z(x)|\, dx < + \infty, \\
&P(v_n(x)) \to v(x) \quad\text{almost everywhere in $\mathbb{R}^N$}.
\end{align*}
Then $\|(P(v_n)-v)z\|_{L^1(B)} \to 0$ for any bounded Borel set $B$.

If we have in addition 
\begin{equation*}
\lim_{s \to 0} \frac{P(s)}{Q(s)} =0
\end{equation*}
and
\begin{equation}\label{eq:12}
\lim_{|x| \to +\infty} \sup_{n} |v_n(x)|=0,
\end{equation} 
then $\|(P(v_n)-v)z\|_{L^1(\mathbb{R}^N)}=0$.
\end{lemma}
Condition (\ref{eq:12}) means that the sequence $\{v_n\}_n$ decays
uniformly to zero at infinity. When working with radially symmetric
$H^1$ functions, this is true by a theorem of Strauss
(\cite{MR0454365}). In fractional Sobolev spaces, the situation is
more complicated. The following theorem is proved (in a more general
setting) in \cite{MR1790248}. See also \cite{MR2902295}.
\begin{theorem} \label{th:decay}
Let $0<p \leq +\infty$.
\begin{itemize}
\item[(i)] Let either $s>1/p$ and $0 < q \leq +\infty$ or $s=1/p$ and
  $0<q \leq 1$. Then there exists a constant $C>0$ such that
\begin{equation*}
|f(x)| \leq C |x|^{\frac{1-N}{p}}\|f\|_{W^{s,p}(\mathbb{R}^N)}
\end{equation*}
for all $f \in W^{s,p}_{\mathrm{rad}}(\mathbb{R}^N)$.
\item[(ii)] Let $(N-1)/N<p$. Furthermore, let either $s<1/p$ and $0<q
  \leq +\infty$ or $s=1/p$ and $1<q \leq +\infty$. Then for all $|x|
  \geq 1$ there exists a sequence $\{f_n\}_n$ of smooth and compactly
  supported radial functions (depending on $x$) such that
  $\|f_n\|_{W^{s,p}(\mathbb{R}^N)} = 1$ and $\lim_{n \to +\infty}
  |f_n(x)|=+\infty$.
\end{itemize}
\end{theorem}
It follows easily from (i) of the previous Theorem that the space
$W^{s,2}_{\mathrm{rad}}(\mathbb{R}^N)$ is compactly embedded into
$L^q(\mathbb{R}^N)$ for any $2<q<2^\star$, provided that $s>1/2$.

However, this embedding is compact for any $0<s<1$, as proved by Lions
(\cite{MR683027}).
\begin{theorem} \label{th:compact}
  Let $N \geq 2$, $s>0$, $p \in [1,+\infty)$; we set $p^\star =
  Np/(N-sp)$ if $sp<N$ and $p^\star = +\infty$ if $sp \geq N$. The
  restriction to $W^{s,p}_{\mathrm{rad}}(\mathbb{R}^N)$ of the
  embedding $W^{s,p}(\mathbb{R}^N) \subset L^q(\mathbb{R}^N)$ is
  compact if $p<q<p^\star$.
\end{theorem}
According to Theorem \ref{th:decay}, part (ii), the proof cannot be based on
\emph{pointwise} estimates at infinity, when $p=2$ and $0<s<1/2$. It is based on some \emph{integral} estimate of the decay at infinity, i.e.
\begin{equation*}
\||x|^{(N-1)/p} u \|_{W^{s,p}} \leq C \|u\|_{W^{s,p}}
\end{equation*}
for any radially symmetric $u \in W^{s,p}(\mathbb{R}^N)$. This is
enough to show the compactness of the embedding, but it is too weak
for a pointwise estimate of the decay of $u$. If $s>1/2$, then
Sobolev's embedding theorem implies that the integral estimate gives
also a pointwise estimate.

\begin{remark}
A ``Strauss-like'' decay lemma is also proved in \cite{Dipierro2012} for \emph{radially decreasing} elements of $H^s(\mathbb{R}^N)$. Needless to say, we cannot use that result in our setting, since we are not allowed to rearrange our functions in a decreasing way.
\end{remark}

\nocite{*}
\bibliographystyle{amsplain}
\bibliography{general}

\end{document}